\documentclass[a4paper, 12pt, reqno]{amsart}
\usepackage[english]{babel}
\usepackage[utf8]{inputenc}
\usepackage{amsthm, amsmath, amssymb, amsrefs, enumerate, enumitem, mathtools, nicefrac, bbm}

\usepackage[top=1.5in, bottom=1.5in, left=1.3in, right=1.3in]{geometry}
\binoppenalty=\maxdimen
\relpenalty=\maxdimen

\usepackage[T1]{fontenc}
\usepackage{libertine}

\usepackage[colorlinks=true,linkcolor=black,citecolor=black,filecolor=black]{hyperref}

\theoremstyle{plain}
\newtheorem{thm}{Theorem}[section]
\newtheorem{prop}[thm]{Proposition}
\newtheorem{lemma}[thm]{Lemma}
\newtheorem{cor}[thm]{Corollary}

\theoremstyle{definition}
\newtheorem{defn}[thm]{Definition}

\theoremstyle{remark}

\usepackage{xpatch}
\makeatletter
\AtBeginDocument{\xpatchcmd{\@thm}{\thm@headpunct{.}}{\thm@headpunct{}}{}{}}
\xpatchcmd{\proof}{\@addpunct{.}}{\@addpunct{:}}{}{}
\makeatother

\makeatletter
\let\@@pmod\pmod
\DeclareRobustCommand{\pmod}{\@ifstar\@pmods\@@pmod}
\def\@pmods#1{\mkern4mu({\operator@font mod}\mkern 6mu#1)}
\makeatother

\def\C{\mathbb{C}}
\def\H{\mathbb{H}}
\newcommand{\Z}{\mathbb{Z}}
\newcommand{\Q}{\mathbb{Q}}
\newcommand{\N}{\mathbb{N}}
\newcommand{\R}{\mathbb{R}}
\newcommand{\tp}{\theta_{\psi}}
\newcommand{\sdf}{\sigma^{\mathrm{sm}}_{\ell}}
\newcommand{\Dc}{\mathcal{D}}
\newcommand{\Pc}{\mathcal{P}}
\newcommand{\slz}{{\text {\rm SL}}_2(\mathbb{Z})}
\newcommand{\Id}{\mathbbm{1}}
\newcommand{\im}{\textnormal{Im}}
\DeclarePairedDelimiter\norm{\lVert}{\rVert}
\DeclarePairedDelimiter\vt{\lvert}{\rvert}

\title{Multidimensional small divisor functions}
\author{Andreas Mono}
\address{Department of Mathematics and Computer Science, Division of Mathematics, University of Cologne, Weyertal 86-90, 50931 Cologne, Germany}
\email{amono@math.uni-koeln.de}

\begin{document}

\begin{abstract}
This is a short note generalizing the construction from \cite{mamoro}, \cite{mertens2019mock} to multi-indices. We recommend to consider both references first. We obtain polar harmonic Maa{\ss} forms of non-positive integral weight if the dimension is even and greater than $2$. We provide explicit examples in dimension $4$, $6$, $8$, and $10$.
\end{abstract}

\maketitle

\section{Introduction - One-dimensional case} \label{sec:onedim}
In a recent paper \cite{mertens2019mock}, Mertens, Ono, and Rolen defined and investigated a new type of mock modular form, whose coefficients are given by a \textit{small divisor function}. We summarize their approach. As usual, we let $\tau = u+iv \in \H$ and $q \coloneqq \mathrm{e}^{2\pi i \tau}$. Let $P_{\ell}\left(\frac{n}{d},d\right) \in \Q[X,Y]$, and $\psi$, $\chi$ be Dirichlet characters of moduli $M_{\psi}$, $M_{\chi}$ respectively. We denote by $\chi_{-4}$ the unique odd Dirichlet character of modulus $4$, and we define 
\begin{align*}
D_n &\coloneqq \left\{ d \mid n \ \colon 1 \leq d \leq \frac{n}{d} \text{ and } d \equiv \frac{n}{d} \pmod*{2}  \right\}, \\
\sdf(n) &\coloneqq \sum_{d \in D_n} \chi\left(\frac{\frac{n}{d}-d}{2}\right) \psi\left(\frac{\frac{n}{d}+d}{2}\right) P_{\ell}\left(\frac{n}{d},d\right).
\end{align*}

Additionally, we require \emph{Shimura's theta-function}
\begin{align*}
\tp(\tau) \coloneqq \frac{1}{2} \sum_{n \in \Z} \psi(n) n^{\lambda_{\psi}} q^{n^2}, \qquad \lambda_{\psi} \coloneqq \frac{1-\psi(-1)}{2},
\end{align*}
and recall that 
\begin{align} \label{eq:tpspaces}
\begin{split}
\tp \in \begin{cases}
M_\frac{1}{2}(\Gamma_0(4M_\psi^2), \psi) & \text{if} \ \lambda_{\psi} = 0, \\
S_\frac{3}{2}(\Gamma_0(4M_\psi^2), \psi\cdot\chi_{-4}) & \text{if} \ \lambda_{\psi} = 1.
\end{cases}
\end{split}
\end{align}

Furthermore, we recall the definition of a harmonic Maa{\ss} form\footnote{Be aware that there is no overall convention which terminology encodes which growth condition.}.
\begin{defn}
Let $k \in \frac{1}{2}\Z$, and choose $N \in \N$ such that $4 \mid N $ whenever $k \not\in \Z$. Let $\phi$ be a Dirichlet character of modulus $N$. 
\begin{enumerate}[label=(\roman*)]
\item A \emph{weight $k$ harmonic Maa{\ss} form on a subgroup $\Gamma_0(N)$ with Nebentypus $\phi$} is any smooth function $f \colon \H \to \C$ satisfying the following three properties:
\begin{enumerate}[label=(\alph*)]
\item For all $\gamma = \left(\begin{smallmatrix} a & b \\ c & d\end{smallmatrix}\right) \in \Gamma_0(N)$ and all $\tau \in \H$ we have
\begin{align*}
f(\tau) = \left(f\vert_k\gamma\right)(\tau) \coloneqq \begin{cases}
\phi(d)^{-1}(c\tau+d)^{-k} f(\gamma\tau) & \text{if} \ k \in \Z, \\
\phi(d)^{-1}\left(\frac{c}{d}\right)\varepsilon_d^{2k}(c\tau+d)^{-k} f(\gamma\tau) & \text{if} \ k \in \frac{1}{2}+\Z,
\end{cases}
\end{align*}
where $\left(\frac{c}{d}\right)$ denotes the extended Legendre symbol, and
\begin{align*}
\varepsilon_d \coloneqq \begin{cases}
1 & \text{if} \ d \equiv 1 \pmod*{4}, \\
i & \text{if} \ d \equiv 3 \pmod*{4}. \\
\end{cases}
\end{align*}
for odd integers $d$.
\item The function $f$ is harmonic with respect to the weight $k$ hyperbolic Laplacian on $\H$, especially
\begin{align*}
0 = \Delta_k f \coloneqq \left(-v^2\left(\frac{\partial^2}{\partial u^2}+\frac{\partial^2}{\partial v^2}\right) + ikv\left(\frac{\partial}{\partial u} + i\frac{\partial}{\partial v}\right)\right) f.
\end{align*}
\item The function $f$ has at most linear exponential growth at all cusps. 
\end{enumerate}
\item A \emph{polar harmonic Maa{\ss} form} is a harmonic Maa{\ss} form with isolated poles on the upper half plane.
\end{enumerate}
\end{defn}

Then the main result of \cite{mertens2019mock} reads as follows.
\begin{thm}[\protect{\cite[Theorem 1.1]{mertens2019mock}}] \label{mormain}
Suppose that $\psi = \chi \neq \Id$, and that $P_2\left(\frac{n}{d},d\right) = d.$ Denote the corresponding small divisor function by $\sigma^{\mathrm{sm}}_{1}$, and by $E_2$ the Eisenstein series
\begin{align*}
E_2(\tau) \coloneqq 1 - 24\sum_{n \geq 1} \Big(\sum_{d \mid n} d \Big) q^n.
\end{align*}
Define
\begin{align*}
\mathcal{E}^+(\tau) &\coloneqq \frac{1}{\tp(\tau)}\left(\alpha_{\psi}E_2(\tau) + \sum_{n \geq 1} \sigma^{\mathrm{sm}}_{1}(n) q^n\right), \\
\mathcal{E}^-(\tau) &\coloneqq (-1)^{\lambda_{\psi}} \frac{(2\pi)^{\lambda_{\psi}-\frac{1}{2}}i}{8\Gamma\left(\frac{1}{2}+\lambda_{\psi}\right)} \int_{-\overline{\tau}}^{i\infty} \frac{\theta_{\overline{\psi}}(w)}{\left(-i(w+\tau)\right)^{\frac{3}{2}-\lambda_{\psi}}} dw,
\end{align*}
where $\alpha_{\psi}$ is an implicit constant depending only on $\psi$ to ensure a certain growth condition. Then the function $\mathcal{E}^+ + \mathcal{E}^-$ is a polar harmonic Maa{\ss} form of weight $\frac{3}{2}-\lambda_{\psi}$ on $\Gamma_0\left(4M_{\psi}^2\right)$ with Nebentypus $\overline{\psi} \cdot \chi_{-4}^{\lambda_{\psi}}$.
\end{thm}
In analogy to the classical divisor sums $\sigma_k(n)$, Mertens, Ono, and Rolen called their function $\mathcal{E}^+$ a \textit{mock modular Eisenstein series with Nebentypus}. Furthermore, they related their result to partition functions for special choices of $\psi$, and proved a $p$-adic property of $\mathcal{E}^+$, compare \cite[Corollary 1.3, Theorem 1.4]{mertens2019mock}.
\newline
\newline
\indent In \cite{mamoro}, Males, Rolen, and the author discovered another example of a polar harmonic Maa{\ss} form adapting the construction from \cite{mertens2019mock}.
\begin{thm}[\protect{\cite[Theorem 1.1, Theorem 1.3]{mamoro}}] \label{mmrmain}
Suppose that $\psi$ is odd, $\chi$ is even, and that $P_2\left(\frac{n}{d},d\right) = d^2$. Denote the corresponding small divisor function by $\sigma^{\mathrm{sm}}_{2}$, and define
\begin{align*}
\mathcal{F}^+(\tau) &\coloneqq \frac{1}{\tp(\tau)} \cdot \begin{cases}
\sum_{n \geq 1} \sigma^{\mathrm{sm}}_{2}(n) q^n & \text{if } \chi \neq \Id, \\
\smallskip
\frac{1}{2}\sum_{n \geq 1}\psi(n)n^2 q^{n^2} + \sum_{n \geq 1} \sigma^{\mathrm{sm}}_{2}(n) q^n & \text{if } \chi = \Id,
\end{cases} \\
\mathcal{F}^-(\tau) &\coloneqq\frac{i}{\pi \sqrt{2}} \int_{-\overline{\tau}}^{i\infty} \frac{\theta_{\chi}(w)}{\left(-i(w+\tau)\right)^{\frac{3}{2}}} dw.
\end{align*}
\begin{enumerate}[label=(\roman*)]
\item If $\chi \neq \Id$ then the function $\mathcal{F}^+ + \mathcal{F}^-$ is a polar harmonic Maa{\ss} form of weight $\frac{3}{2}$ on $\Gamma_0\left(4M_{\chi}^2\right) \cap \Gamma_0\left(4M_{\psi}^2\right)$ with Nebentypus $\overline{\chi} \cdot \left(\psi \cdot \chi_{-4}\right)^{-1}$.
\item If $\chi = \Id$ then the function $\mathcal{F}^+ + \mathcal{F}^-$ is a polar harmonic Maa{\ss} form of weight $\frac{3}{2}$ on $\Gamma_0\left(4M_{\psi}^2\right)$ with Nebentypus $\left(\psi \cdot \chi_{-4}\right)^{-1}$.
\end{enumerate} 
\end{thm}
Moreover, if $\psi = \chi_{-4}$, $\chi = \Id$ , Males, Rolen and the author related $\mathcal{F}^+$ to Hurwitz class numbers, and proved a $p$-adic property of $\mathcal{F}^+ $ in both cases of $\chi$ as well, compare \cite[Corollary 1.6, Theorem 1.8]{mamoro}.
\newline
\newline
\indent The proof of Theorem \ref{mormain} and \ref{mmrmain} is performed in three main steps. To describe them, we let
\begin{align*}
\Gamma(s,z) \coloneqq \int_z^{\infty} t^{s-1}\mathrm{e}^{-t} dt,
\end{align*}
be the \emph{incomplete Gamma function}, which is defined for $\mathrm{Re}(s) > 0$ and $z \in \C$. It can be analytically continued in $s$ via the functional equation
\begin{align*}
\Gamma(s+1,z) = s\Gamma(s,z) + z^{s}\mathrm{e}^{-z},
\end{align*}
and has the asymptotic behavior
\begin{align*}
\Gamma(s,v) \sim v^{s-1}\mathrm{e}^{-v}, \quad \vert v \vert \to \infty 
\end{align*}
for $v \in \R$. In addition, let 
\begin{align*}
\xi_{\kappa} \coloneqq 2iv^{\kappa}\overline{\frac{\partial}{\partial\overline{\tau}}} = iv^{\kappa}\overline{\left(\frac{\partial}{\partial u} + i\frac{\partial}{\partial v}\right)}
\end{align*}
be the \emph{Bruinier--Funke operator of weight $\kappa$}, and
\begin{align*}
\pi_{\kappa}f (\tau) \coloneqq \frac{(\kappa-1) (2i)^{\kappa}}{4\pi} \int_{\H} \frac{f\left(x+iy\right)y^{k}}{\left(\tau-x+iy\right)^{\kappa}} \frac{dxdy}{y^2},
\end{align*}
be the \emph{weight $\kappa$ holomorphic projection operator}, whenever $f$ is translation invariant, and the integral converges absolutely.

Moreover, we let
\begin{align*}
g(\tau) &\coloneqq \sum_{n \geq 1} \beta\left(n\right)q^{n}, \qquad f^+(\tau) \coloneqq \frac{1}{g(\tau)} \sum_{n \geq 1} \sdf(n)q^n, \\
f^-(\tau) &\coloneqq \sum_{m \geq 1} \alpha\left(m\right)m^{k_{f}-1}\Gamma\left(1-k_{f},4\pi m v\right)q^{-m}, \qquad f(\tau) \coloneqq (f^+ +f^-)(\tau).
\end{align*}

Then we proceed as follows.
\begin{enumerate}[label=(\Roman*)]
\item Show that
\begin{align*}
\pi_{\kappa}\left(fg\right)(\tau) = 0.
\end{align*}
To this end, we rewrite the definition of the given non-holomorphic part (see \cite[Lemma 4.1]{mamoro} for instance), and next we utilize the following result. Here and throughout, $\Pc_r^{(a,b)}$ denotes the \emph{Jacobi polynomial} of degree $r$ and parameter $a$, $b$, which we introduce in Section \ref{sec:jac}.
\begin{prop}[\protect{\cite[Proposition 1.7, Corollary 4.2]{mamoro}}] \label{prop:hpcalc}
Let $k_{f} \in \R \setminus \N$, $k_g \in \R \setminus \left(-\N\right)$, such that $\kappa \coloneqq k_{f}+k_g \in \Z_{\geq 2}$. Let $\alpha(m)$, $\beta(n)$ be two complex sequences, and define the functions $f$, $g$ as above. Suppose that 
\begin{enumerate}
\item the function $(fg)(r+iv)$ grows at most polynomially as $v \searrow 0$, where $r \in \Q$, and that
\item the function $(fg)(iv)$ grows at most polynomially as $v \nearrow \infty$.
\end{enumerate}
Then the weight $\kappa$ holomorphic projection of $f^-g$ is given by
\begin{align*}
\pi_{\kappa}\left(f^-g\right)(\tau) &= -\Gamma(1-k_{f}) \sum_{m \geq 1} \sum_{n-m \geq 1} \alpha\left(m\right) \beta\left(n\right) \\
& \hspace{3em} \times \left(n^{k_{f}-1}\Pc_{\kappa-2}^{(1-k_{f},1-\kappa)}\left(1-2\frac{m}{n}\right)-m^{k_{f}-1}\right)q^{n-m}.
\end{align*}
Furthermore, it holds that $\pi_{\kappa}\left(f^+g\right)(\tau) = \left(f^+ g\right)(\tau)$.
\end{prop}
In addition, the holomorphic part $f^+ g$ has to be rewritten as well, see the proof of Theorem \ref{mormain} in \cite[Section 4]{mamoro}.

\item We compute
\begin{align*}
\xi_{\kappa} \left(fg\right)(\tau) = -(4\pi)^{1-k_f} v^{k_{g}} \left(\sum_{m \geq 1} \overline{\alpha(m)} q^m \right) \overline{g(\tau)},
\end{align*}
and choose the coefficients $\alpha\left(m\right)$, $\beta\left(n\right)$, such that this function is modular of weight $2-\kappa$.

\item Conclude that $fg$ is modular of weight $\kappa$ by the following result.
\begin{prop}[\protect{\cite[Proposition 2.3]{mertens2019mock}}]
Let $h \colon \H \to \C$ be a translation invariant function such that $\vert h(\tau) \vert v^{\delta}$ is bounded on $\H$ for some $\delta > 0$. If the weight $k$ holomorphic projection of $h$ vanishes identically for some $k > \delta + 1$ and $\xi_k h$ is modular of weight $2-k$ for some subgroup $\Gamma < \slz$, then $h$ is modular of weight $k$ for $\Gamma$.
\end{prop}
The subtle growth conditions are required to include the case $\pi_2$, and are clearly satisfied if we deal with higher weight holomorphic projections, in which case the integral defining $\pi_k$ converges absolutely. 

Lastly, verify harmonicity and the growth property towards the cusps required by the definition of a harmonic Maa{\ss} form.
\end{enumerate}

Finally, we mention one remark from \cite[p.\ 5]{mamoro}, which states that there are more choices of half integral parameters $k_f$, $k_g$, which lead to other choices of polynomials $P_{\ell}\left(\frac{n}{d},d\right)$ in the definition of $\sdf$, such that step (I) above works. 
\newline
\newline
\indent We refer to the first two sections of \cite{mamoro} for more details, and for overall preliminaries introducing the aforementioned objects together with their key properties.

\section{Statement of the result}
We arrive at the following result by combining the lemmas from the Section \ref{sec:multidim} as outlined during Section \ref{sec:onedim}. The functions  $\sdf$ and $f_{\ell}$ are defined at the beginning of Section \ref{sec:multidim}.

\begin{thm} \label{thm:main}
Let $\psi$ be an odd Dirichlet character, $\chi$ be an even and non-trivial Dirichlet character. Let $\ell \in 2\N+2$. Define $P_{\ell}$ as indicated in Corollay \ref{cor:cond}, obtaining the corresponding small divisor function $\sdf$. Then the resulting function $f_{\ell}$ is a polar harmonic Maa{\ss} form of weight $2-\frac{\ell}{2} \in -\N_0$ on $\Gamma_0(4M_{\chi}^2) \cap \Gamma_0(4 M_{\psi}^2)$ with Nebentypus $\overline{\chi} \cdot \left(\psi \cdot \chi_{-4}\right)^{-1}$. Its shadow is given by a non-zero constant multiple of $\theta_{\overline{\chi}}^{\ell}$. 
\end{thm}

In other words, the technique presented in \cite{mamoro}, \cite{mertens2019mock} applies straightforward in higher even dimensions, except for dimension two. We plan to find and investigate applications of $f_{\ell}$ to other areas of number theory, such as combinatorics, as in the one--dimensional case \cite[Corollary 1.3]{mertens2019mock}.

\section*{Acknowledgement} 
We would like to thank the anonymous referee for many valuable comments on an earlier version of this paper.

\section{Multidimensional Case} \label{sec:multidim}
We fix $\ell \in \N$ throughout. Let $\vec{n} = \left(n_1,\ldots,n_{\ell}\right) \in \N^{\ell}$. We recall the usual multi-index conventions
\begin{align*}
\vec{n}! \coloneqq n_1 n_2 \cdots n_{\ell}, \quad \vt{\vec{n}} \coloneqq n_1 + \ldots + n_{\ell}, \quad \norm{\vec{n}} \coloneqq \sqrt{n_1^2 + \ldots + n_{\ell}^2}.
\end{align*}

We let $\psi \neq \Id$, and consider
\begin{align*}
\tp(\tau)^{\ell} = \sum_{\vec{n} \in \N^{\ell}} \psi\left(\vec{n}!\right)\left(\vec{n}!\right)^{\lambda_{\psi}}q^{\norm{\vec{n}}^2}.
\end{align*}

Moreover, we relax our assumption to $P_{\ell} \in \Q(X,Y)$, and we let
\begin{align*}
\Dc_{\vec{n}} &\coloneqq \bigtimes_{j=1}^{\ell} D_{n_j} \\
&= \left\{ \vec{d} \in \N^{\ell} \ \colon d_j \mid n_j, \ 1 \leq d_j \leq \frac{n_j}{d_j} \text{ , and } d_j \equiv \frac{n_j}{d_j} \pmod*{2} \quad \forall \ 1 \leq j \leq \ell \right\} \\
\sdf\left(\vec{n}\right) &\coloneqq \sum_{\vec{d} \in \Dc_{\vec{n}}} \left\{ \prod_{j=1}^{\ell} \chi\left(\frac{\frac{n_j}{d_j}-d_j}{2}\right) \psi\left(\frac{\frac{n_j}{d_j}+d_j}{2}\right)\left(\frac{\frac{n_j}{d_j}-d_j}{2}\right)^{\lambda_{\chi}} \left(\frac{\frac{n_j}{d_j}+d_j}{2}\right)^{\lambda_{\psi}} \right\} \\
& \hspace{4em} \times P_{\ell}\left(\norm{\left(n_j / d_j \right)_{1\leq j\leq \ell}}^2,\norm{\vec{d}}^2\right).
\end{align*}

Consequently,
\begin{align*}
f_{\ell}^+(\tau) &\coloneqq \frac{1}{\tp(\tau)^{\ell}} \sum_{\vec{n} \in \N^{\ell}} \sdf\left(\vec{n}\right) q^{\vt{\vec{n}}}, \\
f_{\ell}^-(\tau) &\coloneqq \frac{1}{\Gamma\left(1-k_{f_{\ell}}\right)} \sum_{\vec{m} \in \N^{\ell}} \chi\left(\vec{m}!\right)\left(\vec{m}!\right)^{\lambda_{\chi}}\norm{\vec{m}}^{2(k_{f_{\ell}}-1)}\Gamma(1-k_{f_{\ell}},4\pi \norm{\vec{m}}^2 v)q^{-\norm{\vec{m}}^2}, \\
f_{\ell}(\tau) &\coloneqq (f_{\ell}^+ + f_{\ell}^-)(\tau).
\end{align*}

We insert this setting into the constructive method described in the first section, and devote a subsection to each step.

\subsection{First step}
We verify that the first step continues to hold due to exactly the same proofs as in \cite[Section 3]{mamoro}. We have to be careful regarding the summation conditions, which are determined one step after the application of the Lipschitz summantion formula. Explicitly, we obtain
\begin{align*}
\pi_{\kappa}\left({f_{\ell}^-}\tp^{\ell}\right)(\tau) &= -\sum_{r \geq 1}\sum_{\substack{\vec{m},\vec{n} \in \N^{\ell} \\ \norm{\vec{n}}^2-\norm{\vec{m}}^2 = r}} \chi\left(\vec{m}!\right)\left(\vec{m}!\right)^{\lambda_{\chi}} \psi\left(\vec{n}!\right) \left(\vec{n}!\right)^{\lambda_{\psi}} \\
& \hspace{2em} \times \left(\norm{\vec{n}}^{2(k_{f_{\ell}}-1)}\Pc_{\kappa-2}^{(1-k_{f_{\ell}},1-\kappa)}\left(1-2\frac{\norm{\vec{m}}^2}{\norm{\vec{n}}^2}\right)-\norm{\vec{m}}^{2(k_{f_{\ell}}-1)}\right)q^{r}.
\end{align*}

To match this expression with $f_{\ell}^+ g$, we rewrite the small divisor function. We substitute
\begin{align*}
\vec{a} \coloneqq \left(\frac{\frac{n_1}{d_1}+d_1}{2},\ldots,\frac{\frac{n_{\ell}}{d_{\ell}}+d_{\ell}}{2}\right), \qquad \vec{b} \coloneqq \left(\frac{\frac{n_1}{d_1}-d_1}{2},\ldots,\frac{\frac{n_{\ell}}{d_{\ell}}-d_\ell}{2}\right),
\end{align*}
from which we deduce 
\begin{align*}
\vec{d} = \vec{a}-\vec{b}, \qquad \vec{a}+\vec{b} = \left(n_j / d_j \right)_{1\leq j\leq \ell}, \qquad \vt{n} = \norm{\vec{a}}^2-\norm{\vec{b}}^2.
\end{align*}

Thus,
\begin{align*}
f_{\ell}^+\tp^{\ell}(\tau) &= \sum_{\vec{b} \in \N^{\ell}} \sum_{\vec{a}-\vec{b} \in \N^{\ell}} \chi\left(\vec{b}!\right)\left(\vec{b}!\right)^{\lambda_{\chi}} \psi\left(\vec{a}!\right) \left(\vec{a}!\right)^{\lambda_{\psi}} \ P_{\ell}\left(\norm{\vec{a}+\vec{b}},\norm{\vec{a}-\vec{b}}\right) q^{\norm{\vec{a}}^2-\norm{\vec{b}}^2}.
\end{align*}

We transform the summation condition.
\begin{lemma} \label{lem:rewrite}
We have
\begin{align*}
f_{\ell}^+\tp^{\ell}(\tau) = \sum_{r \geq 1}\sum_{\substack{\vec{m},\vec{n} \in \N^{\ell} \\ \norm{\vec{n}}^2-\norm{\vec{m}}^2 = r}} \chi\left(\vec{m}!\right)\left(\vec{m}!\right)^{\lambda_{\chi}} \psi\left(\vec{n}!\right) \left(\vec{n}!\right)^{\lambda_{\psi}} \ P_{\ell}\left(\norm{\vec{m}+\vec{n}},\norm{\vec{m}-\vec{n}}\right) q^{r}.
\end{align*}
\end{lemma}

\begin{proof}
Note that if $\vec{a}-\vec{b} \in \N^{\ell}$, then
\begin{align*}
\norm{\vec{a}}^2-\norm{\vec{b}}^2 = \sum_{j=1}^{\ell} (a_j+b_j)(a_j-b_j) \geq 1.
\end{align*}
Conversely, suppose $\norm{\vec{a}}^2-\norm{\vec{b}}^2 \geq 1$. Recall that $n_j = (a_j+b_j)(a_j-b_j) \in \N$ for every $1 \leq j \leq \ell$ by definition of $f^+$, and $a_j+b_j$ is always positive. Thus, $(a_j-b_j) \geq 1$ for every $1 \leq j \leq \ell$, which proves the lemma.
\end{proof}

Hence, we achieve the following result by virtue of Proposition \ref{prop:hpcalc}.
\begin{cor} \label{cor:cond}
If $P_{\ell}$ is defined by the condition
\begin{align*}
\norm{\vec{b}}^{2(k_{f_{\ell}}-1)}\Pc_{\kappa-2}^{(1-k_{f_{\ell}},1-\kappa)}\left(1-2\frac{\norm{\vec{a}}^2}{\norm{\vec{b}}^2}\right)-\norm{\vec{a}}^{2(k_{f_{\ell}}-1)} = P_{\ell}\left(\norm{\vec{a}+\vec{b}},\norm{\vec{a}-\vec{b}}\right),
\end{align*}
then we have $\pi_{\kappa}\left(f_{\ell}\tp^{\ell}\right)(\tau) = 0$.
\end{cor}

\subsection{Second step}
We summarize the result of a standard calcualtion.
\begin{lemma}
We have 
\begin{align*}
\xi_{\kappa}\left(f_{\ell}\tp^{\ell}\right)(\tau) = -\frac{(4\pi)^{1-k_{f_{\ell}}}}{\Gamma\left(1-k_{f_{\ell}}\right)} v^{k_{\tp^{\ell}}} \theta_{\overline{\chi}}(\tau)^{\ell} \frac{\vt{\tp(\tau)}^{2\ell}}{\tp(\tau)^{\ell}}
\end{align*}
away from the zeros of $\tp$.
\end{lemma}

\begin{proof}
By definition and linearity of $\xi_{\kappa}$, it holds that
\begin{align*}
\xi_{\kappa}\left(f_{\ell}^- \tp^{\ell}\right)(\tau) = (\xi_{\kappa}f_{\ell}^-)(\tau)\cdot\overline{\tp(\tau)^{\ell}} + \overline{f_{\ell}^-(\tau)}\left(\xi_{\kappa}\tp^{\ell}\right)(\tau) = (\xi_{\kappa}f_{\ell}^-)(\tau)\cdot\overline{\tp(\tau)^{\ell}},
\end{align*}
where the last step uses that $\tp^{\ell}$ is holomorphic. Next, one computes\footnote{Compare the proof of \cite[Lemma 2.12]{mamoro} for some intermediate steps.}
\begin{align*}
(\xi_{\kappa}f_{\ell}^-)(\tau) = -\frac{(4\pi)^{1-k_{f_{\ell}}}}{\Gamma\left(1-k_{f_{\ell}}\right)} v^{k_{\tp^{\ell}}} \sum_{\vec{m} \in \N^{\ell}} \overline{\chi\left(\vec{m}!\right)} \left(\vec{m}!\right)^{\lambda_{\chi}} q^{\norm{\vec{m}}^2},
\end{align*}
from which we infer the claim. 
\end{proof}

Combining the previous result with the modularity of Shimura's theta function (see equation \eqref{eq:tpspaces}), and the fact that
\begin{align*}
\im\left(\gamma\tau\right) = \frac{v}{\vt{c\tau+d}^2}
\end{align*}
for every $\gamma = \left(\begin{smallmatrix} a & b \\ c & d\end{smallmatrix}\right) \in \slz$ and every $\tau \in \H$, we obtain the following corollary.
\begin{cor}
If $\chi \neq \Id$ then $\xi_{\kappa}\left(f_{\ell}\tp^{\ell}\right)$ is modular of weight
\begin{align*}
\ell\left(\frac{1}{2}+\lambda_{\overline{\chi}}\right) - \ell\left(\frac{1}{2}+\lambda_{\psi}\right)
\end{align*}
on $\Gamma_0(4M_{\chi}^2) \cap \Gamma_0(4 M_{\psi}^2)$ with Nebentypus $\overline{\chi} \cdot \left(\psi \cdot \chi_{-4}\right)^{-1}$.
\end{cor}

Thus, we stipulate $\psi$ to be odd, and $\chi$ to be even and non-trivial, getting
\begin{align*}
\kappa = 2-(-\ell) \in \Z_{\geq 2}, \qquad k_{f_{\ell}} = 2 - \frac{\ell}{2},
\end{align*}
as desired.

\subsection{Third step}
We verify the two remaining conditions of a polar harmonic Maa{\ss} form.
\begin{lemma} \label{lem:laststep}
Let $\tau \in \H$ with $\tp(\tau) \neq 0$. Then, the function $f_{\ell} = f_{\ell}^+ + f_{\ell}^-$ satisfies
\begin{align*}
0 = \Delta_{k_{f_{\ell}}} f_{\ell},
\end{align*}
and has the required growth property of a polar harmonic Maa{\ss} form.
\end{lemma}

\begin{proof}
The first assertion follows by construction of $f_{\ell}$. Since $\tp^{\ell}$ is of exponential decay towards all cusps, the function $f_{\ell}^+$ admits at most linear exponential growth towards all cusps. In particular, the cusp $i\infty$ is a removable singularity of $f^+$, because both numerator and denominator vanish at $i\infty$ of order $\ell$. In addition, the function $f_{\ell}^-$ decays exponentially towards $i\infty$, since the incomplete Gamma function does (and it dominates the powers of $q$). The transformation behaviour of $\theta_{\chi}$ under the full modular group $\slz$ implies that $f_{\ell}^-$ is of at most moderate growth towards all cusps. Indeed, choosing suitable scaling matrices yields additional factors of polynomial growth inside the Fourier expansion of $f_{\ell}^-$. This establishes the second assertion.
\end{proof}

\subsection{Conclusion}
We justify the application of Proposition \ref{prop:hpcalc}, which proves Theorem \ref{thm:main}.

\begin{proof}[Proof of Theorem \ref{thm:main}]
By definition, the Fourier coefficients of $\tp^{\ell}f_{\ell}^+$ expanded at $i\infty$ are of moderate growth, whence the growth of $\tp^{\ell}f_{\ell}^+$ towards any cusp has to be moderate. Consequently, the growth of $\tp^{\ell}f_{\ell}$ towards any cusp is moderate according to the proof of Lemma \ref{lem:laststep}. Thus, the assumptions in Proposition \ref{prop:hpcalc} are satisfied by $\tp^{\ell}f_{\ell}$. Performing the outlined steps concludes the proof of Theorem \ref{thm:main}.
\end{proof}

\section{Numerical examples}

\subsection{An interlude on Jacobi polynomials} \label{sec:jac}
The Jacobi polynomials $\Pc_r^{(a,b)}$ admit a representation in terms of of Gau{\ss}' hypergeometric function ${}_2F_1$, namely
\begin{align*}
\Pc_{r}^{(a,b)}(z) = \frac{\Gamma(a+r+1)}{r! \ \Gamma(a+1)} {}_2F_1\left(-r,a+b+r+1,a+1,\frac{1-z}{2}\right),
\end{align*}
for any $r \in \N$. This yields many identities between Jacobi polynomials of ``neighboring'' degree $r$ and parameters $a$, $b$, that is $r \in \{r-1,r,r+1\}$ and analogously for $a,b$. For instance, one could use Gau{\ss} contiguous relations, to obtain such identities. 

In particular, this leads to a recursive characterization of the Jacobi polynomials. More precisely, we have
\begin{align*}
\Pc_{0}^{(a,b)}(z) &= 1, \qquad \Pc_{1}^{(a,b)}(z) = \frac{1}{2}\left(a-b+(a+b+2)z\right), \\
c_1(j) \Pc_{j+1}^{(a,b)}(z) &= \left(c_2(j)+c_3(j)z\right)\Pc_{j}^{(a,b)}(z) - c_4(j)\Pc_{j-1}^{(a,b)}(z),
\end{align*}
where
\begin{align*}
c_1(j) &= 2(j+1)(j+a+b+1)(2j+a+b), \qquad c_2(j) = (2j+a+b+1)\left(a^2-b^2\right), \\
c_3(j) &= (2j+a+b)(2j+a+b+1)(2j+a+b+2), \\
c_4(j) &= 2(j+a)(j+b)(2j+a+b+2).
\end{align*}

\subsection{Explicit examples} \label{sec:exa}
Note that the parallelogram law and the fact $\vt{n} = \norm{\vec{a}+\vec{b}} \norm{\vec{a}-\vec{b}}$ yield
\begin{align*}
\norm{\vec{a}}^2 &= \frac{\norm{\vec{a}+\vec{b}}^2+\norm{\vec{a}-\vec{b}}^2}{4} + \frac{\norm{\vec{a}+\vec{b}} \norm{\vec{a}-\vec{b}}}{2}, \\
\norm{\vec{b}}^2 &= \frac{\norm{\vec{a}+\vec{b}}^2+\norm{\vec{a}-\vec{b}}^2}{4} - \frac{\norm{\vec{a}+\vec{b}} \norm{\vec{a}-\vec{b}}}{2}.
\end{align*}

The case $\ell = 2$ has to be excluded since $k_{f_{\ell}} \neq 1$.

\subsubsection{Higher even dimensions}
On one hand, if $\ell = 4$ for instance, we have
\begin{align*}
\kappa = 6, \qquad k_{f_4} = 0, \qquad \frac{\Pc_{4}^{(1,-5)}\left(1-2\frac{\norm{\vec{a}}^2}{\norm{\vec{b}}^2}\right)}{\norm{\vec{b}}^2} - \frac{1}{\norm{\vec{a}}^2} = \frac{\left(\norm{\vec{a}}^2-\norm{\vec{b}}^2\right)^5}{\norm{\vec{a}}^2\norm{\vec{b}}^{10}},
\end{align*} 
and thus, we choose the function $P_{4}$ as
\begin{align*}
P_{4}\left(\norm{\vec{a}+\vec{b}},\norm{\vec{a}-\vec{b}}\right) = \frac{\norm{\vec{a}-\vec{b}}^5 \norm{\vec{a}+\vec{b}}^5}{\left(\frac{\norm{\vec{a}+\vec{b}}^2+\norm{\vec{a}-\vec{b}}^2}{4} + \frac{\norm{\vec{a}+\vec{b}} \norm{\vec{a}-\vec{b}}}{2}\right) \left(\frac{\norm{\vec{a}+\vec{b}}^2+\norm{\vec{a}-\vec{b}}^2}{4} - \frac{\norm{\vec{a}+\vec{b}} \norm{\vec{a}-\vec{b}}}{2}\right)^5}.
\end{align*}
Similarly, we compute (with $x \coloneqq \norm{\vec{a}}$, $y \coloneqq \norm{\vec{b}}$)
\begin{align*}
y^{-4} \ \Pc_{6}^{(2,-7)}\left(1-2\frac{x^2}{y^2}\right) - x^{-4} &= \frac{\left(x^2-y^2\right)^7}{x^4 y^{16}} \left(7x^2+y^2\right), \\
y^{-6} \ \Pc_{8}^{(3,-9)}\left(1-2\frac{x^2}{y^2}\right) - x^{-6} &= \frac{\left(x^2-y^2\right)^9}{x^6 y^{22}} \left(45x^4+9x^2y^2+y^4\right), \\
y^{-8} \ \Pc_{10}^{(4,-11)}\left(1-2\frac{x^2}{y^2}\right) - x^{-8} &= \frac{\left(x^2-y^2\right)^{11}}{x^8 y^{28}} \left(286x^6+66x^4y^2+11x^2y^4+y^6\right),
\end{align*}
from which we read off the corresponding definitions of $P_{\ell}$. 

Because of the aforementioned recursive nature of the Jacobi polynomials, the indicated pattern continues to hold for every even dimension $\ell \in 2\N+2$ by induction.

\subsubsection{Higher odd dimensions}
On the other hand, the case of dimension $\ell \in 2\Z_{\geq 2} - 1$ produces more complicated functions $P_{\ell}$. For example, if $\ell = 3$ we have
\begin{align*}
& \kappa = 5, \qquad k_{f_3} = \frac{1}{2}, \\ 
&{\scriptstyle \frac{\Pc_{3}^{\left(\frac{1}{2},-4\right)}\left(1-2\frac{\norm{\vec{a}}^2}{\norm{\vec{b}}^2}\right)}{\norm{\vec{b}}} - \frac{1}{\norm{\vec{a}}} = -\frac{\left(\norm{\vec{a}}-\norm{\vec{b}}\right)^4\left(5\norm{\vec{a}}^3+20\norm{\vec{a}}^2\norm{\vec{b}}+29\norm{\vec{a}}\norm{\vec{b}}^2+16\norm{\vec{b}}^3\right)}{16\norm{\vec{a}}\norm{\vec{b}}^7} },
\end{align*}
and if $\ell = 5$, we have
\begin{align*}
& \kappa = 7, \qquad k_{f_5} = -\frac{1}{2}, \\ 
& {\scriptstyle \frac{\Pc_{5}^{\left(\frac{3}{2},-6\right)}\left(1-2\frac{\norm{\vec{a}}^2}{\norm{\vec{b}}^2}\right)}{\norm{\vec{b}}^3} - \frac{1}{\norm{\vec{a}}^3} } \\ 
& {\scriptstyle =  \frac{-693\norm{\vec{a}}^{13}+4095\norm{\vec{a}}^{11}\norm{\vec{b}}^2-10010\norm{\vec{a}}^9\norm{\vec{b}}^4+12870\norm{\vec{a}}^7\norm{\vec{b}}^6-9009\norm{\vec{a}}^5\norm{\vec{b}}^8+3003\norm{\vec{a}}^3\norm{\vec{b}}^{10}-256\norm{\vec{b}}^{13}}{256\norm{\vec{a}}^3\norm{\vec{b}}^{13}} }.
\end{align*}
We observe that we are left with odd powers of $\norm{\vec{a}}$, $\norm{\vec{b}}$ in both odd-dimensional cases. If we keep the dependence of $P_{\ell}$ on $\norm{\vec{a}\pm\vec{b}}$, which ultimately justifies the terminology ``divisor function'', then odd powers obstruct a definition of $P_{\ell}$ via the parallelogram law in these cases of $\ell$. Once more, an inductive argument via the recursive characterization of the Jacobi polynomials extends this phenomenon to all odd dimensions $\ell \in 2\N + 1$.

\end{document}